\documentclass{article}
\usepackage{graphics}
\usepackage{ifthen}
\usepackage{amsmath,amsfonts,amstext,amscd,bezier,amssymb,a4,enumerate,epsfig,theorem,psfrag,subfigure}
\usepackage[latin1]{inputenc}
\usepackage{graphicx}

\newcommand{\RR}{{\mathbb{R}}}

\newcommand{\TT}{{\mathbb{S}}}

\newcommand{\Ran}{{\operatorname{Ran}}}

\newtheorem{prop} {Proposition}
\newtheorem{theo} {Theorem}

\def\qed{\unskip\nobreak\hfil\penalty50\hskip1.75em\null\nobreak\hfil
$\blacksquare$ {\parfillskip=0pt \finalhyphendemerits=0 \par}\medbreak}

\newcommand\ran{\operatorname{Ran}}

\title{Fibers and global geometry of functions}
\author{Marta Calanchi, Carlos Tomei and André Zaccur}
\date{}

\begin{document}
\maketitle

\centerline{\it Dedicated to Djairo, an example to follow in many directions}

\begin{abstract}
Since the seminal work of Ambrosetti and Prodi, the study of global folds was enriched by geometric concepts and extensions accomodating new examples. We present the advantages of considering fibers, a construction dating to Berger and Podolak's view of the original theorem. A description of folds in terms of properties of fibers gives new perspective to the usual hypotheses in the subject. The text is intended as a guide, outlining arguments and stating results which will be detailed elsewhere.
\end{abstract}

\medbreak


\smallbreak

{\noindent\bf MSC-class:} 35B32, 35J91, 65N30.

\section{Introduction}\label{sec:intro}

%\begin{center}
%\begin{figure}[htbp]
%\centerline{\includegraphics[width=20cm,height=10cm]{TUDO2reta2.eps}}
%\caption{Em azul: $g(x)=8sech^2(x+1)$. Em vermelho: $g(x)=2sech^2(4(x+1))$. Em preto a soma dos dois.}
%\end{figure}
%\end{center}

When we teach the first courses in calculus and complex or real analysis, a great emphasis is given to geometric issues: we plot graphs, enumerate conformal mappings among special regions, identify homeomorphisms. Alas, this is far from being enough: mappings become too complicated soon. Still, the geometric approach, especially combined with numerical arguments, is very fruitful in some nonlinear contexts.

It is rather surprising that some infinite dimensional maps can be studied in a similar fashion --- one may even think about their graphs! The examples which are amenable to such approach are very few, and they elicit the same sense of wonder that (the equally rare) completely integrable systems do: one is left with a feeling of deep understanding. This text is dedicated to some such examples.

The interested reader could hardly do better than going through the review papers by Church and Timourian (\cite{CT1}, \cite{CT2}), which cover extremely well the material up to the mid nineties. Their approach is strongly influenced by the original Ambrosetti-Prodi view of the problem, which we describe in Section \ref{globalfolds}. In a nutshell, the global geometry of a proper function $F$ is studied through certain properties of its critical set $C$ together with its image $F(C)$, along with the stratification of $C$ in terms of singularities.

This much less ambitious text is mainly an enumeration of techniques and of some recent developments, some of which have not been published. We mostly take the Berger-Podolak route (\cite{BP}) which has been extended by Podolak in \cite{P} and, we believe, still allows for improvement. Instead of the critical set, we concentrate on the restriction of $F$ to appropriate low dimensional manifolds (one dimensional, in the Ambrosetti-Prodi case), the so called fibers.

Essentially, fibers are appropriate in the presence of {\it finite spectral interaction}, which roughly states that the function $F: X \to Y$ splits into a sum of linear and nonlinear terms, $F = L - N$ and $N$ deforms $L$ substantially only along a few eigenvectors spanning a subspace $V \subset X$. The domain splits into orthogonal subspaces, $X = H \oplus V$ and the hypotheses on the nonlinearities are naturally anisotropic. Different requests on $H$ and $V$ yield a global Lyapunov-Schmidt decomposition of $F$: on affine subspaces obtained by translating $H$, $F$ is a homeomorphism and complications due to the nonlinear term manifest on fibers, which are graphs of functions from $V$ to $H$.

Fibers are also convenient for the verification of  properness of $F$. In particular, one may search for folds in nonlinear maps defined on functions with unbounded domains, which are natural in physical situations. Fibers also provide the conceptual starting point for algorithms that solve a class of partial differential equations, an idea originally suggested by Smiley (\cite{S}, \cite{SC}) and later implemented for finite spectral interaction of the Dirichlet Laplacian on rectangles in \cite{CT}.

An abstract setup in the spirit of the characterization of folds as in \cite{CT1}, or like the one we present in Section \ref{sec:adapted}, provide a better understanding of the role of the hypotheses in the fundamental example of Ambrosetti and Prodi. Elliptic theory seems to be less relevant than one might think, it is just that it provides a context in which the required hypotheses are satisfied.

In Section 2, we present the seminal examples --- the Dolph-Hammerstein homeomorphisms and the Ambrosetti-Prodi fold --- in a manner appropriate for our arguments. Fibers and sheets are defined and constructed in Section 3. A global change of coordinates in Section 4 gives rise to adapted coordinates, in which the description of critical points is especially simple. A characterization of the critical points strictly in terms of spectral properties of the Jacobian $DF$ is given. Also, the three natural steps to identify global folds become easy to identify. Further study of how to implement each step is the content of Sections 5, 6 and 7. The last section is dedicated to some examples.

The text is written as a guide: we try to convey the merits of a set of techniques, without providing details. Complete proofs will be presented elsewhere (\cite{CTZ1}, \cite{CTZ2}).

Alas, we stop at folds. There are scattered results in which local or  global cusps were identified: again, the excellent survey \cite{CT2} covers the material up to the mid nineties. So far, the description of cusps seems rather ad hoc. There are characterizations (\cite{CT2}), but they are hard to verify and new ideas are needed. On the other hand, checking that maps are not global folds is rather simple, a matter of showing for example that some points in the image have more than two preimages. A numerical example is exhibited in Section \ref{subsec:numerics}.

\section{The first examples in infinite dimension}

Among the simplest continuous maps between Hilbert spaces are homeomorphisms, in particular linear isomorphisms. A second class of examples are folds.

\subsection{Homeomorphisms: Dolph and Hammerstein}

Dolph and Hammerstein (\cite{D}, \cite{H}) obtained a simple condition under which nonlinear perturbation of linear isomorphisms are still homeomorphisms. A version of their results is the following.

Start with a real Hilbert space $Y$ and a self-adjoint operator $L : X \subset Y \to Y$ for a dense subspace $X$ of $Y$.
Let $\sigma(L)$ be the spectrum of $L$.

\begin{theo} \label{theo:DH} Let $[-c,c] \cap \sigma(L) = \emptyset$ and suppose $N: Y \to Y$ is a Lipschitz map with Lipschitz constant $n < c$.
Equip $X$ with the graph topology, $\| x\|_X = \| x \|_Y + \| Lx\|_Y$. Then the map $F = L - N: X \to Y$ is a Lipschitz homeomorphism.
\end{theo}

Indeed, to solve $F(x) = y$, search for a fixed point of
\[C_y : Y \to Y, \quad C_y(z) =  N ( L^{-1}(z)) +  y \] which is a contraction because the operator $L^{-1}: Y \to Y$ has norm less than $1/c$ by standard spectral theory and then the map $N \circ L^{-1}$ is Lipschitz with constant less than $n/c< 1$ . As usual, the fixed point varies continuously with  $y$. Clearly, $F$ is Lipschitz. To show the same for $F^{-1}$, keep track of the Banach iteration.

\medskip
Notice that the statement allows for differential operators between  Sobolev spaces. Very little is required from the spectrum of $L$.  Clearly, for symmetric bounded operators one should take $X= Y$.

\subsection{Breaking the barrier: the Ambrosetti-Prodi theorem} \label{globalfolds}

What about  more complicated functions? Ambrosetti and Prodi (\cite{AP}) obtained an exquisite example. After refinements by Micheletti and Manes (\cite{MM}), Berger and Podolak (\cite{BP}) and Berger and Church (\cite{BC}), the result may be stated as follows. Let  $\Omega\subset\RR^n$ be a connected, open,  bounded  set with smooth boundary  (for nonsmooth boundaries, see \cite{TZ}). Let $H^2(\Omega)$ and $H^1_0(\Omega)$ be the usual Sobolev spaces and set $X =  H^2(\Omega) \cap H^1_0(\Omega)$ and $Y = H^0(\Omega) = L^2(\Omega)$. The eigenvalues of the Dirichlet Laplacian $-\Delta: X \subset Y \to Y$ are

\[\sigma(- \Delta) =\{0<\lambda_1<\lambda_2\le\ldots\rightarrow\infty\}.\]

Denote by $\phi_1$ the ($L^2$-normalized, positive) eigenvector associated to $\lambda_1$ and split $X  = \ H_X \oplus V_X , Y  = \ H_Y \oplus V_Y  $ in {\it horizontal} and {\it vertical} orthogonal subspaces, where $V_X = V_Y = \langle \phi_1 \rangle$, the one dimensional (real) vector space spanned by $\phi_1$.

\begin{theo} \label{theo:AP}
Let $F: X \to Y$ be $F= L-N$, where $L = - \Delta$, $N(u) = f(u)$, for a smooth, strictly convex function
$f: \RR \to \RR$ satisfying
\[\Ran \ f' = (a,b) \, ,  \quad a < \lambda_1 < b <\lambda_2 \,.\]

Then there are  global homeomorphisms $\zeta:X \to H_Y \oplus \RR  $ and $\xi:Y   \to H_Y \oplus\RR  $ for which $\tilde{F}(z,t) = \xi \circ F \circ \zeta^{-1} (z,t) = (z,-t^2)$.
\end{theo}

Said differently, the following diagram commutes.

\[
\begin{array}{ccc}
X& \stackrel{{\scriptstyle F}}{\longrightarrow}&Y\\
   {\scriptstyle \zeta}\downarrow & &
\downarrow{\scriptstyle \xi}\\
 H_Y \oplus \RR  &\stackrel{{\scriptstyle (z, -t^2)}}{\longrightarrow}&   H_Y \oplus \RR   \\
  \end{array}
\]

Functions which admit such dramatic simplification are called {\it global folds}. The vertical arrows in the diagram above  are (global) changes of variables and sometimes will be $C^1$ maps, but we will not emphasize this point.

\medskip

The original approach by Ambrosetti and Prodi is very geometric (\cite{AP}). In a nutshell, they show that $F$ is a proper map whose critical set $C$ (in the standard sense of differential geometry, the set of points $u \in X$ for which the derivative $DF(u)$ is not invertible) is topologically a hyperplane, together with its image $F(C)$. They then show that $F$ is proper, its restriction to $C$ is injective and $F^{-1}(F(C)) = C$. Finally, they prove that both connected components of $X - C$ are taken injectively to the same component of $Y - F(C)$. Their final result is a counting theorem: the number of preimages under $F$ can only be $0$, $1$ or $2$.

\medskip
Berger and Podolak  (\cite{BP}), on the other hand, construct a global Lyapunov-Schmidt decomposition for $F$. For $V_X= V_Y = \langle \phi_1 \rangle$,  consider {\it affine horizontal (resp. vertical)} subspaces of $X$ (resp. $Y$), i.e., sets of the form $ H_X + t \phi_1 $, for a fixed $t \in \RR$ (resp. $y + V_Y$, for $y \in H_Y$). Let $P: Y \to H_Y$ be the orthogonal projection. The map $P F_t:  H_X  \to H_Y, P F_t(w) = P F(w + t \phi_1)$, is a bi-Lipschitz homeomorphism, as we shall see below. Thus, the inverse under $F$ of vertical lines $y + V_Y$, for $y \in H_Y$ are curves $\alpha_y : \RR \sim V_X \subset X \to H_X$, which we call {\it fibers}. Fibers stratify the domain $X$. Thus, to show that $F$ is a global fold, it suffices to verify that each restriction $F:\alpha_y \to V_Y \sim \RR$, essentially a map from $\RR$ to $\RR$, is a fold.

\medskip
After such a remarkable example, one is tempted to push forward. This is not that simple: if the (generic) nonlinearity $f$ is not convex, there are points in $Y$ with four preimages (\cite{CTZ1}), so the associated map $F:X \to Y$ cannot be a global fold (for a numerical example, see Section  \ref{subsec:numerics}).

\section{Fibers and height functions}\label{sec:fibers}

Fibers come up in \cite{BP} and \cite{SC} for $C^1$ maps associated to second order differential operators  and in \cite{MST2} in the context of first order periodic ordinary differential equations.  Due to the lack of self-adjointness, the construction in \cite{MST2} is of a very different nature. We follow \cite{P} and \cite{TZ}, which handle Lipschitz maps, allowing the use of piecewise linear functions in the Ambrosetti-Prodi scenario, namely $f$ given by ${f'}(x) = a \hbox{ or } b$, depending if $x <0$ or $x>0$ (\cite{CFS}, \cite{LM2}).

Let $X$ and $Y$ be Hilbert spaces, $X$ densely included in $Y$. Let $L: X \subset Y \to Y$  be a self-adjoint operator with
a simple, isolated, eigenvalue $\lambda_p$, with eigenvector $\phi_p \in X$ with $\|\phi_p\|_Y = 1$. Notice that $\lambda_p$ may be located anywhere in the spectrum $\sigma(L)$ of $L$. As before, consider horizontal and vertical orthogonal subspaces,
\[ X  = \ H_X \oplus V_X \ , \ Y  = \ H_Y \oplus V_Y \ , \ \hbox{ for } \ V_X = V_Y = \langle \phi_p \rangle \]
and  the projection $P:Y \to H_Y$. Let $P F_t: H_X \to H_Y$ be the  projection on $H_Y$ of the restriction  of $F$ to the affine subspace $H_X + t \phi_p$,  $P F_t(w) = P F(w+t \phi_p)$. In the same fashion, the nonlinearity  $N : Y \to Y$ gives rise to maps
$PN_t: H_Y \to H_Y$, which we require to be Lipschitz  with constant $n$ independent of $t \in \RR$ so that
\[ [-n,n] \cap \sigma(L) = \{ \lambda_p \} \ . \eqno{(H)} \]

The standard Ambrosetti-Prodi map fits these hypotheses. In this case, $X \subset Y$ are Sobolev spaces and the derivative $f ' : \RR \to \RR$ is bounded by $a$ and $b$. Set
\[ \gamma = (a + b) /2, \quad L = - \Delta - \gamma, \quad N(u) = f(u) - \gamma u\] and $\lambda_p = \lambda_1$, the smallest eigenvalue of $-\Delta$. Then  the Lipschitz constant $n$ of the maps $PN_t$ satisfies $n < \gamma - a = b - \gamma < \lambda_2 - \gamma$, so that $\lambda_1 - \gamma \le n$.

\begin{theo}\label{theo:Fvv}
Let $F: X \to Y$ satisfy $(H)$ above.
Then for each $t \in \RR$, the map $P F_t$ is a bi-Lipschitz homeomorphism, and a $C^k$ diffeomorphism if $F$ is $C^k$. The Lipschitz constants for $P F_t$ and $(P F_t)^{-1}$  are independent of $t$.
\end{theo}

\begin{proof}
The proof follows Theorem \ref{theo:DH} once the potentially nasty eigenvalue $\lambda_p$ is ruled out. Let $c$ be the absolute value of the point in $\sigma(L) \setminus \{ \lambda_p\}$ closest to $0$, so that $0 \le n < c$. The operator $L: X \to Y$ restricts to   $L: H_X \to H_Y$, which is invertible self-adjoint, and again $L^{-1}: H_Y \to H_Y$ with $\| L^{-1} \| \le 1/c$.
The solutions $w \in H_X$ of $P F_t(w) = g \in H_Y$ solve $PLu - PN(u) = Lw - PN_t(w) =g$ for $u = w + t \phi_p$. The solutions $w$  correspond to the fixed points of $C_g : H_Y \to H_Y$, where
\[  \quad C_g(z) =  PN_t(L^{-1} z) +  g, \quad \hbox{ for } \  Lw = z \in H_Y. \]
 The map $C_g$ is a contraction with constant bounded by $ n/c < 1$ (independent of $t$). Now follow the proof of Theorem \ref{theo:DH}. \qed
\end{proof}

The attentive reader may have noticed that the effect of the nonlinearity $N$ along the vertical direction is irrelevant for the construction of fibers.

\medskip
The same construction applies when the interval $[-n,n]$ defined by the Lipschitz constant $n$ of $PN_t: H_X \to H_Y$ interacts with an isolated subset $I$ of $\sigma(L)$ --- more precisely, $I = [-n,n] \cap \sigma(L)$ and there is an open neighborhood $U$ of $I \subset \RR$ for which $I = U \cap \sigma(L)$. In this case $P$ is the orthogonal projection on $I$, which takes into account possible multiplicities. In the special situation when $I$ consists of a finite number of eigenvalues (accounting multiplicity), we refer to {\it finite spectral interaction} between $L$ and $N$.

We concentrate on the case when $I = \{ \lambda_p \}$ consists of a simple eigenvalue.
A more careful inspection of the constants in the Banach iteration in the proof above yields the following result (\cite{CT}, \cite{TZ}). The image under $F$ of horizontal affine subspaces of $X$ are {\it sheets}. The inverse under $F$ of vertical lines of $Y$ are {\it fibers}.

\begin{prop} \label{prop:fibers}  If $F$ is $C^1$, sheets  are graphs of $C^1$ maps from $H_Y$ to $\langle \phi_p \rangle$ and fibers are graphs of $C^1$ maps from  $\langle \phi_p \rangle$ to $H_X$.
 Sheets are essentially flat, fibers are essentially steep.
\end{prop}

We define what we mean by essential flatness and steepness. Let  $\nu(y)$ be the normal at a point $y \in Y$ of (the tangent space of) a sheet, and $\tau(u)$ be the tangent vector at $ u \in X$  of a fiber. Then there is a constant $\epsilon \in (0 , \pi/2)$ such that $\phi_p$ makes an angle less than $\epsilon $ (or greater than $\pi - \epsilon$, due to orientation) with both vectors.

\section{Adapted coordinates and a plan}\label{sec:adapted}

Suppose $L$ and $N$ interact at a simple eigenvalue $\lambda_p$.  Write
\[ F(u) = PF (u) + \langle F(u), \phi_p \rangle \phi_p = PF (u) + h(u) \phi_p \]
where the map $h: X \to \RR $ is called the {\it height function}.
In the diagram below, invertible maps are  bi-Lipschitz (\cite{TZ}) or $C^k$ diffeomorphisms, depending if $PF_t$ is Lipschitz or $C^k$. The smoothness of $h$ and $h^a = h \circ \Phi$ follow accordingly.
\[
\begin{array}{ccl}
{X = H_X \oplus V_X}&
\stackrel{{\scriptstyle F}}{\longrightarrow}&
{Y = H_Y \oplus V_Y} \\
   {\scriptstyle \Phi^{-1}=( P F_t, Id)}\searrow &
      &
\nearrow{\scriptstyle F^a=F \circ \Phi=(Id, h^a)}\\
    &{Y}

&  \\
  \end{array}
\]

The map $F$ has been put in {\it adapted coordinates} by the change of variables $\Phi$:
\[ F^a: Y \to Y  \ , \quad (z,t) \mapsto (z, h^a(z,t)) \ . \]

Notice that fibers of $F$ are taken to vertical lines in the domain of $F^a = F \circ \Phi$. Explicitly, the vertical lines $\{ (z_0, t) : t \in \RR \}$ parameterized by $z_0 \in H_Y$ correspond to fibers $u(z_0,t) = (P F_t)^{-1}(z_0) + t \phi_p = w(z_0,t) + t \phi_p $.
Thus $F^a$ is just a rank one nonlinear perturbation:
\[ F^a(z,t) = (z, h^a(z,t)) \sim z + h^a(z+ t \phi_p) \phi_p \ . \]

 In a very strict sense, this is also true of $F$. In order to make $F$ similar to an Ambrosetti-Prodi map, define $G = F^a \circ (- \Delta): X \to Y$:
\[ u \stackrel{- \Delta}{\longmapsto} z + t \phi_1 \stackrel{{\scriptstyle F^a}}{\longmapsto} z + t\phi_1 + (h^a(z+ t \phi_1)-t) \phi_1 = - \Delta u + \psi(u) \phi_1, \]
for some nonlinear functional $\psi$. We generalize slightly.

\begin{prop} \label{prop:rankone} Let $N$ be a $C^1$ map. Say $L$ and $N$ interact at a simple eigenvalue $\lambda_p$ and $L$ is invertible. Then, after a $C^1$  change of variables, the $C^1$ function $F= L-N: X \to Y$ becomes
$G : X \to Y$, $G = L + \psi(u) \phi_p$, for some $\psi: X \to \RR$.
\end{prop}

\medskip
For Ambrosetti-Prodi operators $F(u) = - \Delta u - f(u)$,  the nonlinear perturbation is given by a Nemitskii map $u \mapsto f(u)$.
It is not surprising that once we enlarge the set of nonlinearities new global folds arise. For a map $F$ given in adapted coordinates by $F^a(z,t ) = (z , h^a(z,t))$, appropriate choices of the {\it adapted height function} $h^a$ yields all sorts of behavior.

The critical set of $F:X \to Y$ is compatible with fibers as follows (\cite{BP}, \cite{CTZ2}).

\begin{prop} \label{prop:Cc} Suppose the $C^1$ map $F:X \to Y$ admits  fibers. Then $u_0$ is a critical point of $F$ if and only if it is a critical point of the height function $h$ along its fiber, or  equivalently of the adapted height function $h^a$.
\end{prop}

Isolated local extrema have to alternate between maxima and minima. In particular, given the appropriate  behavior at infinity at each fiber and the fact that all critical points are of the same type, we learn from a continuity argument that the full critical set $C$ is connected, with a single point on each fiber (\cite{CDT}).

\medskip
The study of a function $F: X \to Y$ reduces to three steps:

 \begin{enumerate}
 \item Stratify $X$ into fibers.
 \item Verify the asymptotic behavior of $F$ along fibers.
 \item Classify the critical points of the restriction of $F$ along fibers.
 \end{enumerate}

The following result is natural from this point of view (\cite{CTZ2}). Let $F:X \to Y$ satisfies $(H)$ of Section \ref{sec:fibers}, so that, by Proposition \ref{prop:fibers}, $X$ stratifies in one dimensional fibers $\{u (z, t ) : t \in \RR\}$, one for each $z \in H_Y$.

\begin{prop}  \label{prop:folds}
Suppose that, on each fiber,
\[ \lim_{t \to \pm \infty} \langle F( u (z, t )), \phi_p \rangle = \lim_{t \to \pm \infty} h(u (z, t )) = - \infty \ . \]
Suppose also that each critical point of $h$ restricted to each fiber is an isolated local maximum. Then $F: X \to Y$ is a global fold, in the sense that there are homeomorphisms on domain and image that give rise to a diagram as in Theorem \ref{theo:AP}.
\end{prop}

To verify that such limits exist, one might check hypotheses $(V\pm)$ in Section \ref{subsec:asy}, but there are alternatives. Similarly, there are ways of obtaining fibers which do not fit the construction presented in Section \ref{sec:fibers} (this is the case for perturbations of non-self-adjoint operators, Section \ref{subsec:nonself}). The upshot is that there is some loss in formulating the three step recipe into a clear cut theorem.

\medskip

 As trivial examples, $h^a(z,t)= - t^2$ is a global fold, whereas $h^a(z,t) = t^3 - t$ has a critical set consisting of two connected components having only (local) folds (from Section \ref{subsec:Morin}). More complicated singularities require the dependence on $z$: not every fiber of $F$ (equivalently, vertical line in the domain of $F^a$) has the same number of critical points close to a cusp, for example. The reader is invited to check that $(z,t) \mapsto (z, t^3 - \langle z, \tilde{\phi} \rangle t)$ is a global cusp, for $\tilde{\phi}$ any fixed vector in $H_Y$. Higher order Morin singularities, considered in Section \ref{sec:singularities}, are obtained in a similar fashion. From the Proposition \ref{prop:rankone}, changes of variables on such maps yield nonlinear rank one perturbations of the Laplacian which are globally diffeomorphic to the standard normal forms of Morin singularities.

\medskip
We consider the standard Ambrosetti-Prodi scenario in the light of this strategy. For the function $F(u) = - \Delta u - f(u)$ defined in Theorem \ref{theo:AP}, elliptic theory yields all sort of benefits --- the smallest eigenvalue of the Jacobian $DF(u)$ is always simple, the ground state may be taken to be a positive function in $X$.

The hypotheses required for the construction of fibers in Theorem \ref{theo:Fvv} do not imply the simplicity of the relevant eigenvalue: there are examples for which there is no naturally defined $C^1$ functional $\lambda_p: X \to \RR$ because two eigenvalues collide. One might circumvent this difficulty by forcing the nonlinearity $N$ to be smaller, but it turns out that this is not necessary. The hypotheses instead imply the simplicity of $\lambda_p$ in an open neighborhood of the critical set $C$ of $F$, and this is all we need, as we shall see in Section \ref{spectralsing}.

The positivity of the ground state and the convexity of the nonlinearity $f$ are used in a combined fashion in the Ambrosetti-Prodi theorem to prove that along fibers the height function only has local maxima. Clearly, this is a property only of critical points. On the other hand, the nonlinearity $N(u) = f(u)$ is so rigid that the standard hypothesis of convexity of $f$ is essentially necessary, as shown in \cite{CTZ1}. More general nonlinearities require a better understanding of the singularities.

\medskip
We now provide more technical details on each of the three steps.

\section{Obtaining fibers in other contexts}\label{sec:fifibers}

For starters, what if $L$ is not self-adjoint, or $X$ is not Hilbert?

\subsection{Podolak's approach} \label{subsec:podolak}
Suppose momentarily that $X$ and $Y$ are Banach spaces. Let $L: X \to Y$ be a Fredholm operator of index zero with kernel generated by  a vector $\phi_X$ and let $\phi_Y$ be a vector not in $\ran L$. Podolak (\cite{P}) considered the following scenario, for which she obtained a lower bound on the number of preimages for a region of $Y$ of vectors with very negative component along $\phi_Y$. Split $X = H_X \oplus V_X$ where $V_X = \langle \phi_X \rangle$ and $H_X$ is any complement. Also, split $Y = H_Y \oplus V_Y$ where  $H_Y= \ran L$ and $V_Y = \langle \phi_Y \rangle$.  In particular $L: H_X \to H_Y$ is an isomorphism.
Also, define the associated projection $P: Y \to H_Y$.  Write
$u = w + t \phi_X, \ y = g + s \phi_Y$ for $w \in H_X$. The equation $F(u) = Lu - N(u) = y$ becomes
\[ L(w + t \phi_X) - N(w + t \phi_X) =  L w - N(w + t \phi_X) = g + s \phi_Y, \]
and, as in Theorem \ref{theo:Fvv}, we are reduced to solving the map
\[ C_g : H_Y \to H_Y, \quad C_g(z) =  PN_t(L^{-1} z) +  g \ , \quad \hbox{ for } \  Lw = z \in H_Y. \]
Her hypotheses imply that such maps are contractions.

\subsection{Transplanting fibers}

The estimates arising from spectral theorem in the Hilbert context are easy to obtain and possibly more effective. Podolak's hypotheses are harder to verify. There is a possibility: getting fibers in Hilbert spaces and transplanting them to Banach spaces. This happens for example when moving from the Ambrosetti-Prodi example as a map between Sobolev spaces (\cite{BP}) to a map between H\"older spaces (\cite{AP}). The classification of singularities is simpler with additional smoothness (Section \ref{sec:singularities}).

\begin{prop} Let $F= L - N: X \to Y$ satisfy  hypothesis $(H)$ of Section \ref{sec:fibers}. Consider the densely included Banach spaces $A \subset X$ and $B \subset Y$ allowing for the $C^1$  restriction $F: A \to B$ for which $V_X = V_Y \subset A$. Suppose that $DF(a): A \to B$ is a Fredholm operator of index zero for each $a \in A$. Then fibers of $F: X \to Y$ either belong to $A$ or do not intersect $A$. \end{prop}

Said differently, if a point $u \in X$ belongs to $A$ then the whole fiber does.

\medskip

In the Ambrosetti-Prodi scenario, this proposition seems to be a consequence of elliptic regularity, which may be used to prove it. Regularity of eigenfunctions is irrelevant: fibers are the orbits of the vector field of their tangent vectors, which are inverses of the vertical vector under $DF(u)$, and necessarily lie in $A$ (\cite{CTZ2}). Tangent vectors are indeed eigenfunctions $\phi_p(u)$ of $DF(u)$ at critical points $u$.

\medskip

The fact that sheets and fibers are uniformly flat and steep (Proposition \ref{prop:fibers}) allows one to modify vertical spaces ever slightly and still obtain space decompositions for which the Lyapunov-Schmidt decomposition, and hence the construction of fibers in Theorem \ref{theo:Fvv}, apply. In particular, transplants may be performed even when the eigenvector $\phi_p$ originally used to define the vertical spaces $V_X= V_Y$ do not have regularity, i.e., do not belong to $A \subset X$. We only have to require that $A$ is dense in $X$, so that $\phi_p$ can be well approximated by a new vertical direction.

\subsection{Fibers and Numerics} \label{subsec:numerics}
Finite spectral interaction is a very convenient context for numerics.
Any question related to solving $F(u)= g$ for some fixed $g \in Y$ reduces to a finite dimensional problem in situations of finite spectral interaction, irrespective of additional hypotheses. If the interaction involves a simple eigenvalue $\lambda_p$, one simply has to look at the restriction of $F$ to the (one dimensional) fiber associated to the affine vertical line through $g$.

Smiley and Chun realized the implications of this fact for numerical analysis (\cite{S}, \cite{SC}). An  implementation for functions $F(u) = - \Delta u - f(u)$ defined on rectangles $\Omega \subset \RR^2$ was presented in \cite{CT}. In the forecoming sections, we will require more stringent hypotheses with the scope of obtaining very well behaved functions $F$ --- we will mostly be interested in global folds. Such additional restrictions might improve on computations, but so far this has not seen to lead to substantial improvements on the available algorithms.

\begin{figure}[ht]
\begin{center}
$\begin{array}{cc}
\epsfig{height=45mm,file=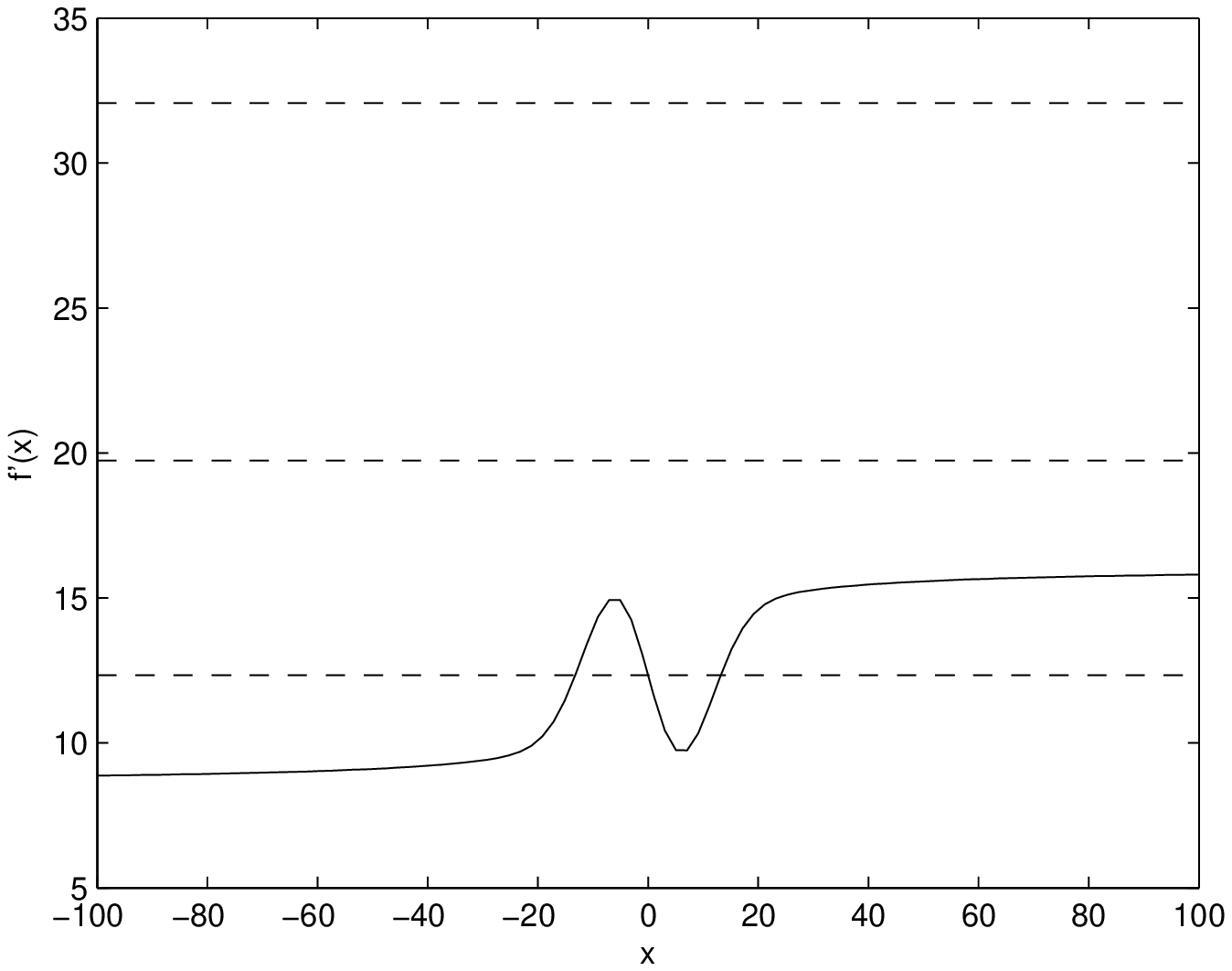} &
\epsfig{height=45mm,file=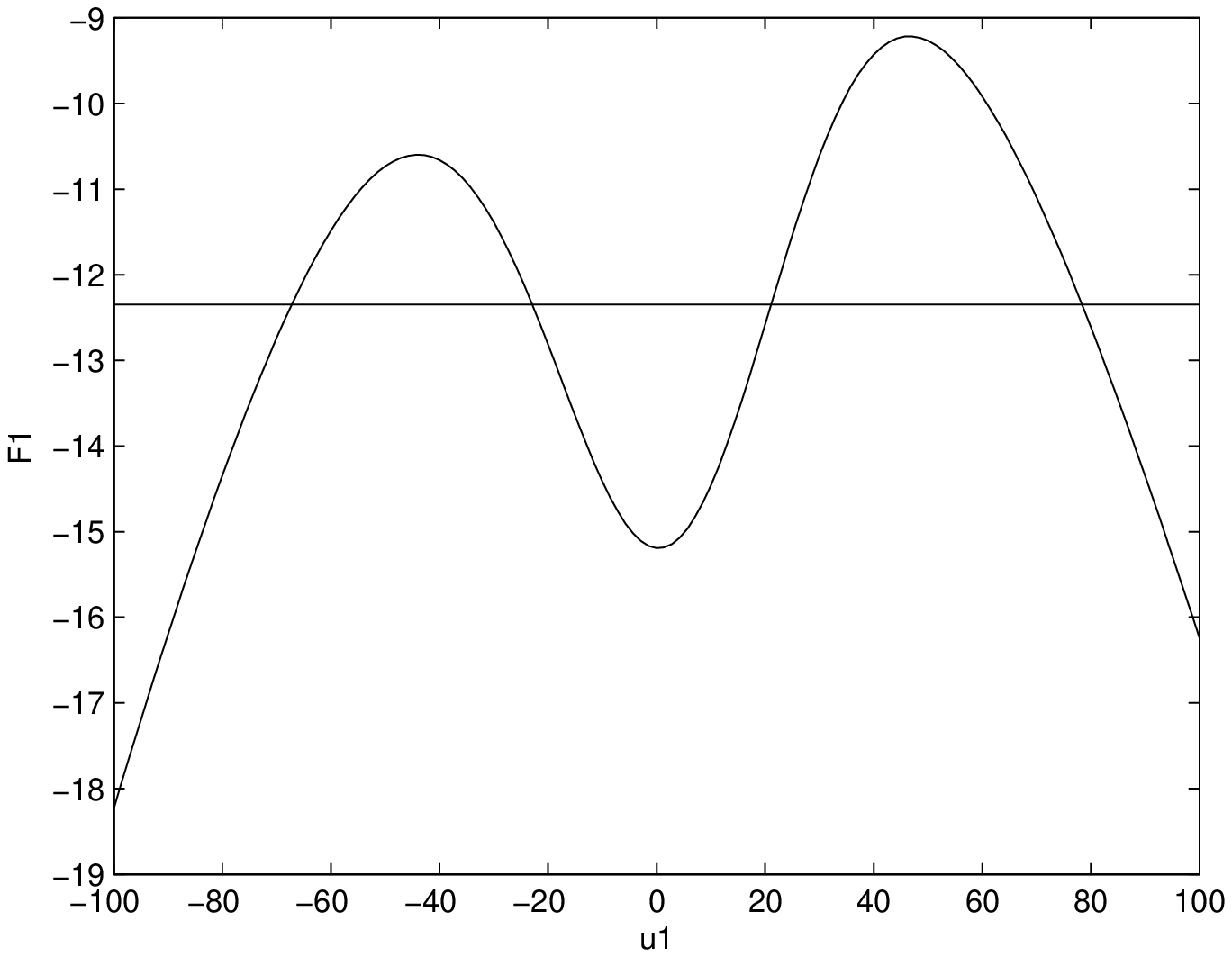}
\end{array}$
\end{center}
%\caption{\capsize The change of variables}
\label{fig:wil}
\end{figure}

We present an example obtained from programs by José Cal Neto (\cite{CT}) and Otavio Kaminski. For $\Omega = [0,1] \times [0,2]$,  $\lambda_1 \sim 12.337$ and $\lambda_2 \sim 19.739$. Consider
\[ - u_{xx} - u_{yy} - f(u) = g \ , \quad (x,y) \in \Omega \ , \quad u = 0 \ \hbox{ in } \  \partial \Omega \ , \]
\[ f'(x) = \frac{\lambda_2 - \lambda_1}{\pi} \ \big( \arctan (\frac{x}{10}) - \frac{2}{5}\  x \ e^{- (x/10)^2} \big) + \lambda_1 \ , \quad f(0) \sim 47.12 \]
\[ g(x,y) = - 100 \big( x(x-1)y^2(y-2) \big) - 35 \sin(\pi x) \sin(\frac{\pi y}{2}) \ . \]
On the left, we show the graphs of $f'$, which interacts only with $\lambda_1$. On the right, the height function $h$ associated to the fiber obtained by inverting the vertical line through $g$. The height value $-12.3$ is reached by four preimages, displayed below. Notice the cameo appearance of the maximum principle: the four graphs sit one on top of the other as one goes up along the fiber (this is very specific of interactions with $\lambda_1$ of the Laplacian with Dirichlet conditions).

\begin{figure}[ht]
\begin{center}
$\begin{array}{cc}
\epsfig{height=45mm,file=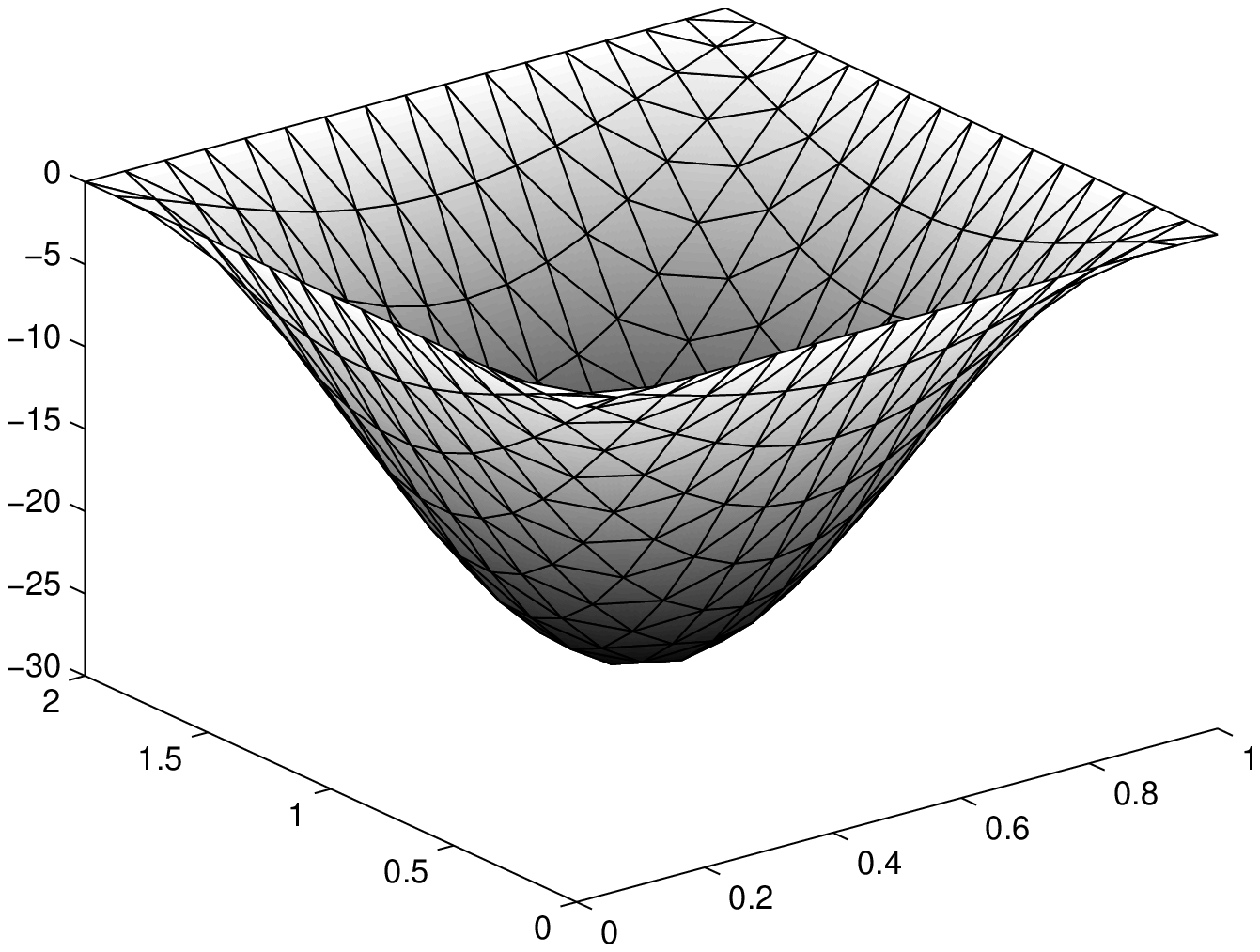} &
\epsfig{height=45mm,file=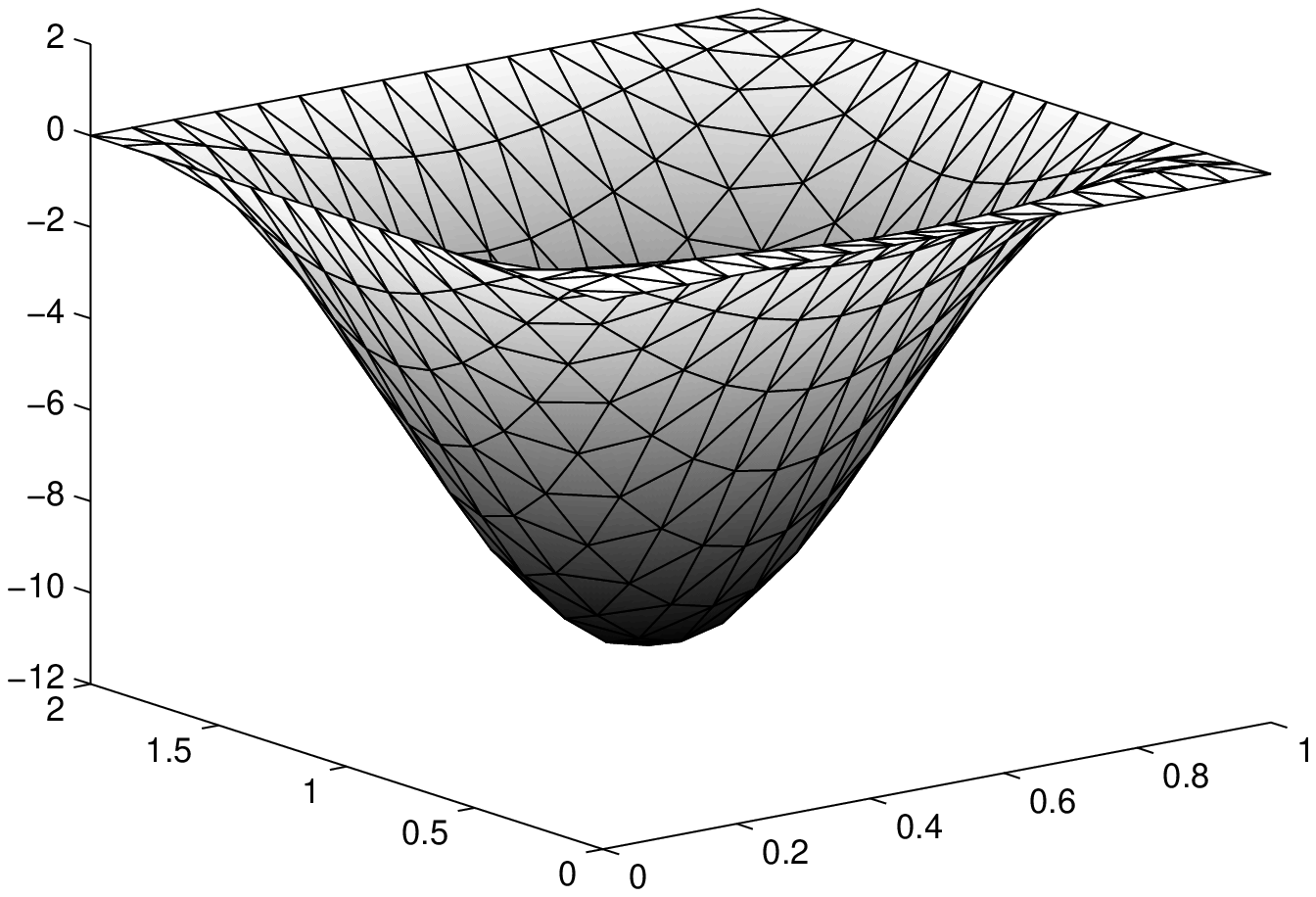} \\
\epsfig{height=45mm,file=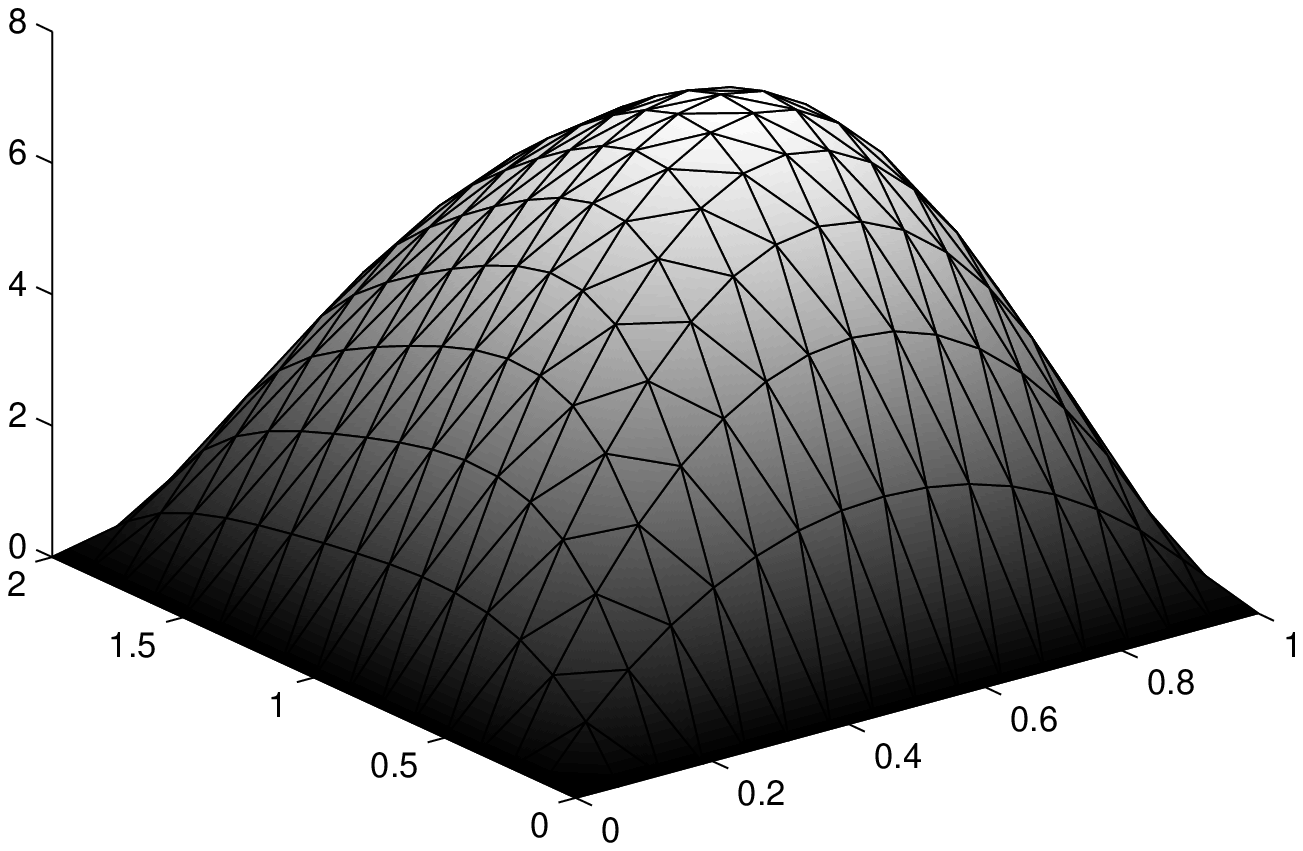} &
\epsfig{height=45mm,file=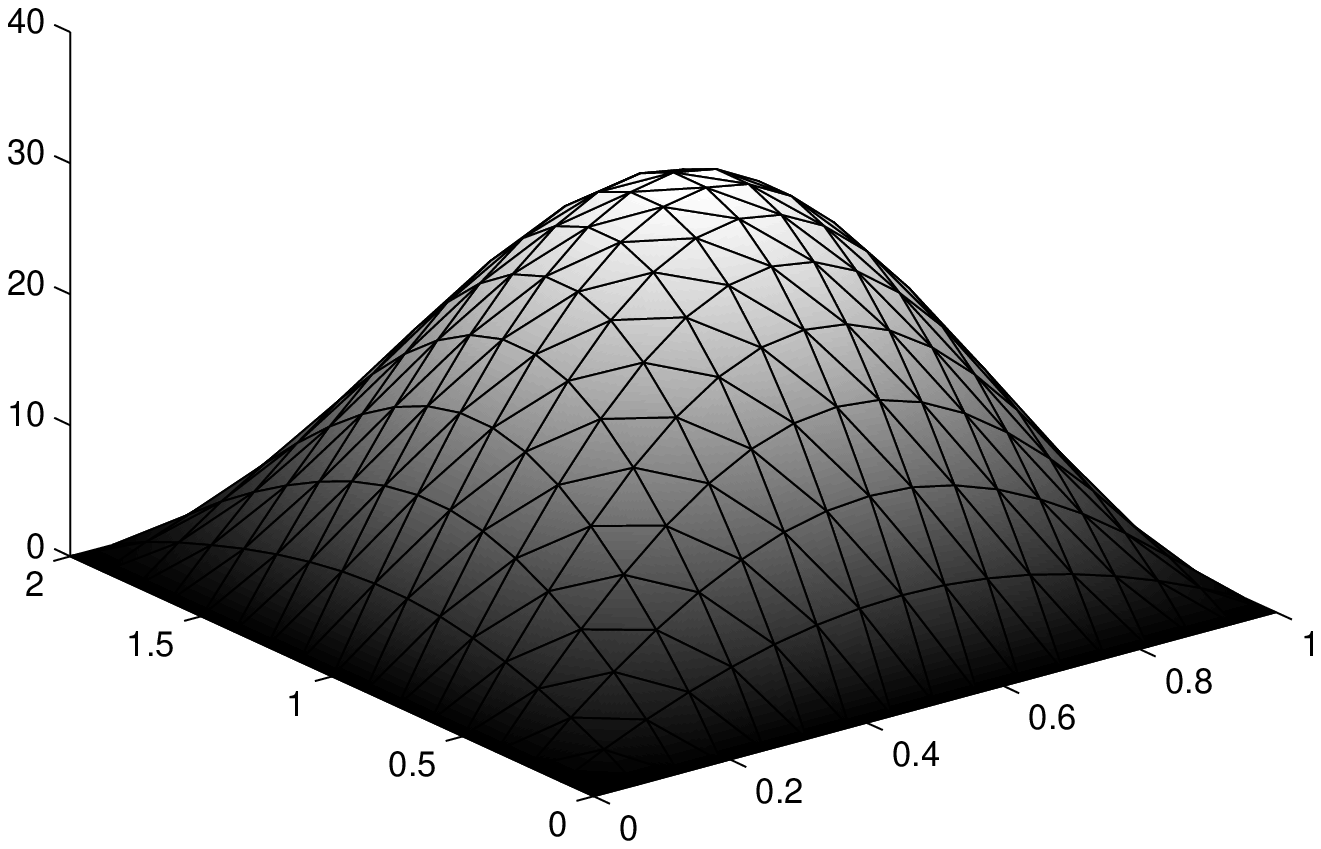}
\end{array}$
\end{center}
%\caption{\capsize The change of variables}
\label{fig:wil}
\end{figure}

%\begin{figure}[ht]
%\begin{center}$
%\begin{array}{cc}
%
%\end{array}$
%\end{center}
%%\caption{\capsize The change of variables}
%\label{fig:wil}
%\end{figure}

\section{Asymptotics of $F$ on fibers and vertical lines}

We stick to one dimensional fibers and consider two issues.
\begin{enumerate}
\item How does $F$ behave at infinity along fibers?
\item How do fibers look like at infinity ?
\end{enumerate}

The first question, to say the very least, is tantamount to characterizing the image of $F$. The second is not relevant for the theoretical study of the global geometry of $F$, since a (global) coordinate system leading to a normal form (like $(z,t) \mapsto (z, - t^2)) $ is insensitive to the shape of fibers. On the other hand, for numerical purposes, a uniform behavior at infinity of the fibers is informative.

\subsection{ $F$ along fibers} \label{subsec:asy}
The inverse of a vertical line $z_0 + V_Y, z_0 \in H_Y$ is the fiber $u(z_0,t) = w(z_0,t) + t \phi_p$:
\[ F(u(z_0,t)) =   z_0 + h^a(z_0,t)\, \phi_p \ .  \eqno{(\ast)} \]
For a fixed $z_0 \in H_Y$, the $C^1$ map $ t\mapsto h^a (z_0,t)$
is the {\it adapted height function} of the fiber associated to $z_0$.
Clearly,
\[   h^a(z_0, t) =
\langle F(u(z_0, t)), \phi_p \rangle =
\langle L(w(z_0,t) + t\phi_p) - N(u(z_0,t)) , \phi_p\rangle  \]
so that
\[ h^a(z_0, t)  = \lambda_p t - \langle N(u(z_0,t)), \phi_p \rangle.
\]

In order to have
\[ \lim_{t \to \pm \infty} \langle F( u (z_0, t )), \phi_p \rangle = \lim_{t \to \pm \infty} h^a(z_0, t ) = - \infty  \]
and some uniformity convenient to obtain properness as discussed in Section \ref{subsec:proper}, we require an extra hypothesis:

\centerline{
 For each $z_0 \in X$, there is a ball $U(z_0) \subset X$ and $\epsilon, T > 0, c_\pm$ such that, for $z \in U(z_0)$,}
 \[ \langle N(u(z,t)), \phi_p \rangle > (\lambda_p  + \epsilon) t  + c_+\ , \quad \hbox{ for } t  > \,\ T \ ,  \eqno{(V+)} \]
 \[ \langle N(u(z,t)), \phi_p \rangle > (\lambda_p  - \epsilon) t  + c_- \ , \quad \hbox{ for } t  <  - T  \, . \eqno{(V-)} \]

 Notice that the asymptotic behavior on each fiber is the same.

\subsection{Asymptotic geometry of fibers}

Again, parameterize fibers as $u(z,t)=w(z,t) + t\phi_p$.
Under mild hypotheses, the vectors $w(z,t)/t$ have a limit for $t \to \pm \infty$, which is independent of $z$. A version of this result was originally obtained by Podolak (\cite{P}).

\begin{prop}
Suppose that $F: X \to Y$, $F = L - N$ satisfies hypothesis $(H)$ of Section \ref{sec:fibers}.  Suppose also that, for every $u\in X$,
$$\lim_{t\to +\infty}\frac{PN(tu)}{t} = N_\infty(u) \in Y.$$
Then there exist $w_{+}, w_{-} \in H_X$ such that, for every fiber $u(z,t)=w(z,t) + t\phi_p$,
$$\lim_{t\to +\infty}\|\frac{w(z,t)}{t} - w_+\|_X  =  0 \ , \quad \lim_{t\to -\infty}\|\frac{w(z,t)}{t} - w_-\|_X = 0$$
which are respectively the unique solutions of the equations
$$Lw - PN_\infty(w + \phi_p) = 0\ , \quad Lw + PN_\infty(-w - \phi_p) = 0.$$
\end{prop}

It turns out that $N_\infty = PN_\infty$ satisfies the same Lipschitz bound that the functions $PN_t$ in Theorem \ref{theo:Fvv}, which is  why both equations are (uniquely) solvable.

Fibers are asymptotically vertical if and only if  $\lim_{|t| \to \infty}  w(z,t)/t = 0$, or equivalently, $PN_\infty(\pm \phi_p) =0$. Indeed, in this case, $w=0$ is the unique solution of both equations. This is what happens in the Ambrosetti-Prodi scenario, where $PN_\infty(u) = (b-\gamma) Pu^+ - (a - \gamma) Pu^-$ (recall $u = u^+ - u^-$), since $\phi_p=\phi_1 >0$.

\subsection{Comparing $F$ on  fibers and on vertical lines}

One might wish to relate the heights of $F$ along fibers and vertical lines, which are easier to handle. In \cite{P} Podolak presented a scenario in which this is possible. We state a version of her result for the case $t \to + \infty$.

\begin{theo}
Let $X\subset Y$ be Hilbert spaces with $X$ dense in $Y$. Let $L:X \to Y$ be a self-adjoint operator with $0 \in \sigma(L)$, a simple, isolated eigenvalue, associated to the normalized kernel vector $\phi_p$. Set $H_Y = \langle \phi_p \rangle^{\perp}$.
Take $N:Y \to Y$ and $F= L-N: X \to Y$ so that
\begin{enumerate}
\item $\|N(u) - N(u_0)\|_Y \leq \epsilon \|u-u_0\|_Y \ , \quad \lim_{t\to +\infty}N(tu)/t = N_\infty(u)$
\item $\langle N_\infty(\phi_p),\phi_p \rangle = - \lim_{t\to +\infty} \langle F(t\phi_p),\phi_p\rangle/ t >0$
 \item $\epsilon \ \|{ \big( L|_{H_Y}} \big)^{-1}\|< 1/2 \ , \quad \epsilon^2 \ \|{ \big( L|_{H_Y}} \big)^{-1}\| < 1/2 \ \langle N_\infty(\phi_p),\phi_p \rangle $\ .
\end{enumerate}
Then, for each fiber $(z_0,t)$ in adapted coordinates,
\[ \big| \ \lim_{t\to +\infty}\frac{h^a(z_0,t)}{t} - \langle N_\infty(\phi_p),\phi_p \rangle \big| \ < \langle N_\infty(\phi_p),\phi_p \rangle.\]

\end{theo}

The number $\langle N_\infty(\phi_p),\phi_p \rangle$ gives the asymptotic behaviour of the height of $F$ along the vertical line through the origin.
The theorem implies that $F$ along the upper part of each fiber converges to the same infinity that $F$ along $\{ t \phi_p, \ t \ge 0 \}$.

A context in which these hypotheses apply is the  Ambrosetti-Prodi operator with a piecewise nonlinearity $f(u) = (\lambda_p + c)u^+ - (\lambda_p - c)u^-$ for a sufficiently small number $c>0$. However,  for pairs $(\lambda_p - c_1, \lambda_p + c_2), p \ne 1$ in the  Fu\v{c}ik spectrum of the (Dirichlet) negative second derivative, for which necessarily $c_1 \ne c_2$ (near $\lambda_p$), the condition involving $\epsilon^2$ does not hold and indeed the thesis is not true.

\subsection{Fibers and the properness of $F$} \label{subsec:proper}

From a more theoretical point of view, fibers circumvent the fundamental issue of deciding if $F$ is proper. For example (\cite{MST2}), the map
\[ F: C^1(\TT^1) \to C^0 (\TT^1), \quad u \mapsto u ' + \arctan (u) \]
is a diffeomorphism from the domain to the open region between two parallel planes,
\[ \big\{  \ y \in C^0 (\TT^1) \ ,  \ - \pi^2 < \int_0^{2\pi} y(\theta) d \theta < \pi^2 \ \big\} \ . \]
Indeed, fibers in this case are simply lines parallel to the vertical line of constant functions, and each is taken to such region.

Perhaps, it would be more appropriate to think of fibers as a tool to show properness (\cite{CTZ2}). As far as we know, for the Ambrosetti-Prodi map $F: X \to Y$ in unbounded domains, the properness has been proved only by making use of fibers (see Section \ref{sec:examples}).

\begin{prop} The map $F: X \to Y$ satisfying hypotheses $(H)$ of Section \ref{sec:fibers} and $(V\pm)$ above is proper if and only if the restriction of $F$ to each fiber is proper.
\end{prop}

Points in the Fu\v{c}ik spectrum of the (Dirichlet) second derivative give rise to maps $F$ which take  the half-fiber $\{ u(0,t), t \ge 0\}$ to a single point $0$ (\cite{TZ}), which shows that $F$ is not proper, although the image of every vertical line has its vertical component taken to infinity.

A possible definition of a topological degree for $F$ becomes innocuous --- the relevant information is essentially the asymptotic behavior of $F$ along each fiber.

\section{Singularities} \label{sec:singularities}

Generic singularities both of $F$ and of each height function are very special --- they are Morin singularities. Morin classified generic singularities of
functions from $\RR^n$ to $\RR^n$ whose derivative at the singularity has one dimensional kernel (\cite{M}). This is sufficient for the study of critical points of height functions on one dimensional fibers, by Proposition \ref{prop:Cc}. In order to do the same for the critical points of the whole function $F:X \to Y$,
we need an equivalent classification
for singularities of functions between infinite-dimensional spaces, which is  very similar (\cite{CDT}, \cite{MST2}, \cite{R}) --- this is how we proceed next.

\subsection{Morin theory in adapted coordinates } \label{subsec:Morin}

The first step in Morin's proof makes use of the implicit function
theorem to write such a singularity at a point $(z_0, t_0)$
in  adapted coordinates, as in Section \ref{sec:adapted}:
\[ F^a: Y = H_Y \oplus V_Y \to Y = H_Y \oplus V_Y, \quad (z,t) \mapsto (z, h^a(z,t)). \]
Say $F^a$ is $C^{k+1}$. The point $(z_0, t_0)$ is
a {\it Morin singularity of  order} $k$  if and only if
\begin{enumerate}
\item $D_t h^a(z_0, t_0) = \cdots = D_t^k h^a(z_0, t_0) = 0$,
$D_t^{k+1} h^a(z_0, t_0) \ne 0$.
\item  The Jacobian
$D(h^a, D_t h^a, \ldots, D_t^{k-1} h^a)(z_0, t_0)$ has maximum rank.
\end{enumerate}
Then, in a neighborhood of $(z_0, t_0)$ there is an additional change of variables which converts $F^a$ to the normal form
\[ (\tilde z, x, t) \mapsto (\tilde z, x, t^{k+1} + x_1 t^{k-1} + \cdots + x_{k-1} t ) \ . \]
Here the coordinates $(\tilde z, x)$ correspond to an appropriate splitting of $Y = \tilde Y \oplus \RR^{k-1}$.

Morin singularities of order 1, 2, 3 and 4 are called, respectively, folds, cusps, swallowtails and butterflies.

Thus, the classification of critical points of $F$ boils down to the study of a family of one dimensional maps, the height functions restricted on fibers. The first requirement  is specific to each fiber (i.e., one checks it for every fixed $z$ near $z_0$), whereas the second relates nearby fibers, i.e., one has to change $z$. Folds are structurally simpler than deeper singularities: the behavior along fibers near a fold point is always the same --- essentially like $t \mapsto -t^2$, whereas this is not the case for cusps, where close to $t \mapsto t^3$ one finds $t \mapsto t^3 \pm \epsilon t$.

\medskip
There is something unsatisfying in the fact that the relevant properties of the critical points of $F$ requires knowledge of some version of the height function. This is circumvented by the next result (\cite{CTZ2}).

\begin{prop} \label{spectralsing}
Suppose $F: X \to Y$ is $C^{k+1}$ and admits one dimensional fibers.
Then there is an open neighborhood $U$ of the critical set $C$ with the  properties below.
\begin{enumerate}
\item There is a unique $C^k$ map $\lambda_p: U \to \RR$ for which $\lambda_p = 0$ on $C$ and is an eigenvalue of $DF$ elsewhere.
\item There is a strictly positive $C^k$ function $p : U \to \RR^+$ such that
\[ \lambda_p (u(z, t)) = p(u(z,t)) \ D_t h(u(z, t)) \ ,  \quad u(z,t) \in U \ . \]
\end{enumerate}

A point $u_0 = u(z_0,t_0)$ is a Morin singularity of order $k$ of $F$ if and only if
\begin{enumerate}
\item $ \lambda_p(u_0) = \cdots = D_t^{k-1} \lambda_p(u_0) = 0$  ,
$D_t^{k} \lambda_p(u_0) \ne 0$ ,
\item  The image of
$D(\lambda_p, \ldots, D_t^{k-2} \lambda_p)(u_0)$ together with $D_t \lambda_p(u_0)$ span $\RR^n$.
\end{enumerate}
\end{prop}

There is an analogous characterization in adapted coordinates.

\subsection{Critical points of the height function}

Consider a critical point $u_0 \in C \subset X$ and the fiber $u(z_0, t)$ through it, $u(z_0, t_0 ) = u_0$. From Proposition \ref{spectralsing}, $u_0$ is a (topological) fold of the height function $h$ restricted to the fiber if and only if $u_0$ is a topologically simple root of $\lambda_p(u)$ along the fiber, i.e., $\lambda_p$ is strictly negative on one side of $u_0$ and strictly positive on the other.

Once we reduce the issue to checking an eigenvalue along a fiber, {\it derivatives are irrelevant}: just study the quadratic form of the Jacobian. Clearly, this only handles topological equivalence between the function and a fold.

More explicitly, in standard Ambrosetti-Prodi contexts, $\lambda_1(u_0)$ is the minimum value of the quadratic form $\langle DF(u_0) v, v \rangle$. The derivative $D_t u(z_0,t_0)$ of the ($C^1$) fiber  is the  eigenfunction $\phi_1(u_0) >0$, and it is easy to check that $\lambda_1$ increases with $t$ by the convexity of the nonlinearity $f$. This should be compared with differentiability arguments, which require some estimate on $\phi_1(u_0)$ (say, boundedness).

The fact that all critical points are local maxima for height functions on fibers, as required in Proposition \ref{prop:folds}, suggest hypotheses to be checked only on the critical set of $F$. This is not the case in the original Ambrosetti-Prodi theorem: the statement of the theorem has the merit that it makes no reference to the critical set at all, an object which in principle is hard to identify. The convexity of the nonlinearity handles the difficulty and, rather surprisingly, is essentially necessary (\cite{CTZ1}). Further examples yielding local maximality are somewhat contrived.

\section{Some examples} \label{sec:examples}

\subsection{The non-autonomous case}

The geometric formulation $F = L - N$ is not sufficient to accomodate situations of the form $F(u(x)) = - \Delta u(x) + f(x, u(x))$, the so called non- autonomous case.
Hammerstein (\cite{H}) had already considered  homeomorphisms of that form. A possibility is requiring that  $X$ and $Y$ are function spaces defined on a domain $\Omega$, so that the variable $x$ makes sense. The formalism above carries over to this scenario without surprises.

More precisely, as usual $X$ and $Y$ are Hilbert spaces, $X$ dense in $Y$. The linear operator $L: X \subset Y \to Y$ is self-adjoint with a simple eigenvalue $\lambda_p$ associated to a normalized eigenvector $\phi_p$. Let $P: Y \to H_Y = \langle  \phi_p \rangle^\perp $ be the orthogonal projection.

From the nonlinear term $N: \Omega \times Y \to Y$, define as before $PN_t: H_Y \to H_Y, t \in \RR$ and require a Lipschitz estimate,
\[ \| PN_t(x, w_1) - PN_t(x,w_0) \|_Y \le n \|w_1 - w_0\|_Y, \quad \hbox{ for } w_0, w_1 \in H_Y , \]
so that $[-n,n] \cap \sigma(L) = \{ \lambda_p \}$, which is the same hypothesis $(H)$ in Section \ref{sec:fibers}. This obtains fibers for $F: X \to Y$ as in Theorem \ref{theo:Fvv}, which satisfy the same properties as those in the autonomous case, in particular, Proposition \ref{prop:fibers}.

The hypothesis which obtain appropriate asymptotic behavior of $F$ along fibers are the obvious counterparts of $(V+)$ and $(V-)$ in Section \ref{subsec:asy}. For the classification of critical points, we simply do not distinguish between the autonomous and non-autonomous case: the subject has become a geometric issue.

\subsection{Schr\"odinger operators on $\RR^n$}\label{subsec:physics}

As was surely known by Ambrosetti and Prodi (and \cite{B} is an interesting example), the Laplacian with Dirichlet conditions might be replaced by more general self-adjoint  operators. The approach in this text is flexible enough to handle nonlinear perturbations of Schr\"odinger operators on unbounded domains yielding global folds. In our knowledge there are no similar results in the literature. Tehrani (\cite{Te}) obtained counting results for Schr\"odinger operators in $\RR^n$ in the spirit of those obtained by Podolak (\cite{P}), indicated in Section \ref{subsec:podolak} .

We state the by now natural hypotheses. Here $Y = L^2(\RR^n)$.

\begin{enumerate}
\item The free operator $ T = -\Delta + v(x): X \subset Y \to Y$ is self-adjoint, with simple, isolated, smallest eigenvalue $\lambda_1$ and positive ground state $\phi_1$.
\item $F: X \subset Y \to Y, F(u) = Tu - f(u)$ is a $C^1$ map.
\item The function $f \in C^2(\RR)$ satisfies $f(0)=0$, $M \geq f''>0$,
$f'(\RR)=(a,b)$ and $ a < \lambda_1 < b  < \min\{ \sigma(T) \setminus \{ \lambda_1\} \}$.
\item The Jacobians $DF(u): X \to Y$ are self-adjoint operators  with eigenpair $(\lambda_1(u),\phi_1(u))$ sharing the properties of $(\lambda_1,\phi_1)$.
\end{enumerate}

\begin{theo}\label{theo:APBS}
Under these hypotheses, the map $F:X \to Y$ is a global fold.
\end{theo}

Such hypotheses are satisfied for $v(x) = x^2/2$, the one dimensional quantum harmonic oscillator, as well as for the hydrogen atom in $\RR^3$, for which $v(x) = -1/|x|$.

Hypotheses on the potential of a Schr\"odinger operator in order to obtain such properties are commonly studied in mathematical physics. The interested reader might consider \cite{BS}, \cite{LL}, \cite{RS}. More about this in \cite{CTZ2}.

\subsection{Perturbations of compact operators} \label{subsec:compact}

We recall Mandhyan's second example of a global fold (\cite{M2}), or better, a special case of the extension given by Church and Timourian (\cite{CT1}).

For $\Omega \subset \RR^n$ a compact subset, let $X= C^0(\Omega)$ and define the compact operator
\[ K: X \to X, \quad K(u)(x) =  \int_\Omega k(x,y) u(y) dy \]
 where the kernel $k \in C^0(\Omega \times \Omega)$ is symmetric and positive.  Let $\mu_1 > \mu_2$ be the largest eigenvalues of $K$.
Now let $f: \RR \to \RR$ be a strictly convex $C^2$ function satisfying
\[ 0 < \lim_{x \to - \infty} f' (x) < 1/\mu_1  < \lim_{x \to  \infty} f' (x) < 1/|\mu_2| \ .\]

\begin{theo} Under these hypotheses for $K$ and $f$, the map
\[ G: X \to X, \quad G(u)(x)= u(x) - K f(u(y))   \]
is a global fold.
\end{theo}

This is the kind of nonlinear map obtained if one started from the Ambrosetti-Prodi original operator $F(u) = - \Delta u - f(u)$ and inverted the Laplacian. Actually, one could take another track: instead of inverting the linear part, one might consider the inversion of the nonlinear map $u \mapsto f(u)$, since $f' $ is bounded away from zero. For maps $G(u) = Ku - f(u)$ obtained this way, we handle the case when $K$ is a general compact symmetric operator $K$.

More precisely, let $\Omega \subset \RR^n$, $B= C^0(\Omega)$ and $ Y = L^2( \Omega)$.
Let $K:B \to B$ and $K:Y\to Y$ be compact operators which preserve the cone of positive functions. Also,  $K:Y\to Y$ has simple largest eigenvalue $\lambda_p = \|K \|$ and second largest eigenvalue $\lambda_s$ . Let $f: \RR \to \RR$ be a strictly convex $C^2$ function, with $f(0)=0$ if $\Omega$ is unbounded. Suppose
\[               \lambda_s < a = \lim_{t \to - \infty} f'(t) < \lambda_p < b = \lim_{t \to \infty} f'(t) \ . \]

\begin{theo}\label{theo:APC}
The map $ F: B \to B \ ,  F(u) = Ku - f(u) $
is a global fold.
\end{theo}

The reader should notice that $F$ is Lipschitz but not  differentiable as a map from $L^2(\Omega)$ to itself.  Still, the direct construction of fibers in $C^0(\Omega)$ is not a simple matter, because properness of $F$ is not immediate. Transplanting fibers in this example is convenient, and was also used in Mandhyan's context.

\subsection{Folds as perturbations of non-self-adjoint operators} \label{subsec:nonself}

McKean and Scovel (\cite{McKS}, \cite{CT1}) studied the Riccati-like map on  functions
\[ u \in L^2([0,1]) \mapsto u + (D_2)^{-1} f(u) \in L^2([0,1]), \quad f(x) = x^2 / 2 , \]
where $(D_2)^{-1}$ is the inverse of the second derivative acting on $W^{1,2}([0,1])$ and showed that the critical set consists of a countable union of (topological) hyperplanes. Church and Timourian (\cite{CT})  showed that the restriction of such map to a neighborhood of one specific critical component is (after global homeomorphic change of variables) a fold. The techniques employed are in the spirit of the original Ambrosetti-Prodi paper.

Fibers were relevant in (\cite{MST2}), where perturbations of first order differential equations (clearly, non-self-adjoint operators) were shown to be global folds. An example is the map on periodic functions with (generic) convex nonlinearities $f$,
\[ F: C^1 (\TT^1) \to C^0(\TT^1), \quad u \mapsto u ' + f(u) \ . \]
McKean and Scovel (\cite{McKS}) and Kappeler and Topalov (\cite{KT}) considered the same map among Sobolev spaces, the celebrated {\it Miura map}, used as a change of variables between the Korteweg-deVries equation and its so called modified version.

More recently, a perturbation of a non-self-adjoint elliptic operator (as in \cite{BNV}, but with Lipschitz boundary) has been shown to yield a global fold (\cite{STZ}).

\subsection{Acknowledgements}

The first author thanks the Departamento de Matemática, PUC-Rio, for its warm hospitality. The second and third authors gratefully acknowledge support from CAPES, CNPq and FAPERJ. We thank José Cal Neto and Otavio Kaminski for the numerical examples.

\bigskip\bigskip\bigbreak

{

\parindent=0pt
\parskip=0pt
\obeylines

\bigskip

Marta Calanchi, Dipartimento di Matematica, Università di Milano,
Via Saldini 50, 20133 Milano, Italia

\smallskip

Carlos Tomei and André Zaccur, Departamento de Matem\'atica, PUC-Rio,
R. Mq. de S. Vicente 225, Rio de Janeiro, RJ 22453-900, Brazil

\bigskip

marta.calanchi@unimi.it
carlos.tomei@gmail.com
zaccur.andre@gmail.com
}


\begin{thebibliography}{[10]}
\bibitem{AP}{ A. Ambrosetti and G. Prodi,
{On the inversion of some differentiable mappings with singularities between Banach spaces},
Ann. Mat. Pura Appl. 93 (1972) 231-246.}
\bibitem{B}{H. Berestycki, { Le nombre des solutions de certain problèmes sémi-linéaires elliptiques}, J. Funct. Anal. 40 (1981), 1-29.}
\bibitem{BNV}{H. Berestycki, L. Nirenberg and S.R.S. Varadhan, {The principal eigenvalue and maximum principle for second-order elliptic operators in general domains}, Comm. Pure Appl. Math. 47 (1994) 47-92.}
\bibitem{BP}{ M.S. Berger and E. Podolak,
{On the solutions of a nonlinear Dirichlet problem},
Indiana Univ. Math. J. 24 (1974) 837-846.}
\bibitem{BC}{ M.S. Berger and P.T. Church,{  Complete integrability and perturbation of a nonlinear Dirichlet problem. I.}, Indiana Univ. Math. J. 28 (1979), 935-952.}
\bibitem{BCT}{ M.S. Berger, P.T. Church and J.G. Timourian, { Folds and cusps in Banach spaces, with applications to nonlinear partial differential equations. I.}, Indiana Univ. Math. J. 34 (1985), 1-19.}
\bibitem{BS}{ F.A. Berezin and M.A. Shubin, {The Schr\"odinger Equation} (1991), Kluwer, Dordrecht.}
\bibitem{CT}{ J.T. Cal Neto and C. Tomei,
{ Numerical analysis of semilinear elliptic equations with finite spectral interaction}, J.Math.Anal.Appl. 395 (2012) 63-77.}
\bibitem{CTZ}{M. Calanchi, C. Tomei and A. Zaccur, {Fibers and global folds in infinite dimension}, in preparation.}
\bibitem{ChT}{P.T. Church and J.G. Timourian, {Global fold maps in differential and
integral equations}, Nonlinear Anal. 18 (1992) , 743-758.}
\bibitem{CDT}{ P.T. Church, E.N. Dancer and J.G. Timourian,
{ The structure of a nonlinear elliptic operator}, Trans. Amer. Math. Soc. 338 (1993) 1-42.}
\bibitem{CT1}{P.T. Church and J.G. Timourian, {Global Structure for Nonlinear Operators
in Differential and Integral Equations I. Folds}, Topological Nonlinear Analysis, II, Progr. Nonlinear Differential Equations Appl., 27, Birkhäuser, Boston, MA, (1997)  109-160.}
\bibitem{CT2}{P.T. Church and J.G. Timourian, {Global Structure for Nonlinear Operators
in Differential and Integral Equations II.Cusps}, Topological Nonlinear Analysis, II, Progr. Nonlinear Differential Equations Appl., 27, Birkhäuser, Boston, MA, (1997)  161-246.}
\bibitem{CFS}{D.G. Costa, D.G. Figueiredo and P.N. Srikanth,
{The exact number of solutions for a class of ordinary differential equations through Morse index computation},
J. Diff. Eqns. 96 (1992) 185-199.}
\bibitem{D}{ C.L. Dolph, { Nonlinear integral equations of the Hammerstein type}, Trans. AMS 66 (1949) 289-307.}
%\bibitem{G} P. Grisvard, Elliptic problems in non smooth domains, Monographs and Studies in Mathematics, 24, Pitman, Boston, 1985.
\bibitem{H}{ A. Hammerstein, { Nichtlineare Integralgleichungen nebst
Anwendungen}, Acta Math. 54 (1929) 117-176.}
\bibitem{K}{ T. Kato, {Perturbation Theory for Linear Operators}, Classics in Mathematics (1980), Springer, New York.}
\bibitem{KT}{T. Kappeler and P. Topalov,
{Global fold structure of the Miura map on $L^2(\TT)$},
Int. Math. Res. Not. 39 (2004), 2039-2068.}
%\bibitem{LM1}{ A.C. Lazer and P.J. McKenna, { On a conjecture related to the number of solutions of a nonlinear Dirichlet problem}, Proc. R. Soc. Edinb. 95A (1983) 275-283.}
\bibitem{LM2}{ A.C. Lazer and P.J. McKenna, { On the number of solutions of a nonlinear Dirichlet problem}, J. Math Anal. Appl. 84 (1981) 282-294.}
\bibitem{LL}{ E.H. Lieb and M. Loss, {Analysis}, Graduate Studies in Mathematics 14 (2001), AMS, Providence.}
\bibitem{M1} {I. Mandhyan, { Examples of global normal forms for some simple nonlinear integral operators}, Nonlinear Anal. 13 (1989),1057-1066.}
\bibitem{M2} {I. Mandhyan, { The diagonalization and computation of some nonlinear integral operators}, Nonlinear Anal. 23 (1994), 447-466.}
\bibitem{McKS} {H.P. McKean and J.C. Scovel, { Geometry of some simple nonlinear differential operators}, Ann. Scuola Norm. Sup. Pisa Cl. Sci. 13 (1986), 299-346.}
\bibitem{MM}{ A. Manes and A.M. Micheletti, { Un'estensione della teoria variazionale classica degli autovalori per operatori ellittici del secondo ordine}, Boll. Unione Mat. Ital.. 7 (1973) 285-301.}
%\bibitem{MST1}{I. Malta, N.C. Saldanha and C. Tomei, { Regular Levels of Nemytskii
%operators are hyperplanes}, J. of Funct. Anal., 143 (1997) 461-469.}
\bibitem{MST2}{I. Malta, N. Saldanha and C. Tomei, { Morin singularities and global geometry in a class of ordinary differential operators}, Topol. Meth. Nonlin. Anal. 10, 137-169, 1997.}
\bibitem{M}{B. Morin, {Formes canoniques des singularités d'une application différentiable}, C. R. Acad. Sci. 260 (1965), 5662-5665, 6503-6506.}
\bibitem{P} {E. Podolak,{ On the range of operator equations with an asymptotically nonlinear term}, Indiana Univ. Math. J. 25, 12 (1976) 1127-1137.}
\bibitem{R}{B. Ruf, {Singularity theory and bifurcation phenomena in differential equations}, Topological Nonlinear Analysis, II, Progr. Nonlinear Diff. Eqs. Appl., 27, Birkhäuser, Boston, MA, (1997)  315-395.}
\bibitem{RS}{M. Reed and B. Simon, {Fourier Analysis, Self-Adjointness},Academic Press, New York (1975).}
%\bibitem{S}{ G. Savaré, {Regularity results for elliptic equations in Lipschitz domains}, J. Funct. Anal. 152 (1998) 176-201.}
\bibitem{STZ}{ B. Sirakov, C. Tomei and A. Zaccur, {Global folds given by perturbations of non-self-adjoint elliptic operators}, in preparation.}
\bibitem{S}{M.W. Smiley, {A finite element method for computing the bifurcation function
for semilinear elliptic BVPs}, J. Comput. Appl. Math., 70 (1996) 311- 327. }
\bibitem{SC}{ M.W. Smiley and C. Chun, {Approximation of the bifurcation equation function for elliptic boundary value problems}, Numer. Meth. Partial Diff. Eqs. 16 (2000) 194-213.}
\bibitem{Te}{H.T. Tehrani, {A multiplicity result for the jumping nonlinearity problem}, J. Diff. Eqs. 188 (2003) 272-305.}
\bibitem{TZ}{C. Tomei and A. Zaccur, {Geometric aspects of Ambrosetti-Prodi operators with Lipschitz nonlinearities}, Analysis and topology in nonlinear differential equations, Progr. Nonlinear Differential Equations Appl., 85, Springer, Cham, (2014)  445-456.}
\end{thebibliography}
\end{document}